\documentclass[12pt,a4paper]{article}

\usepackage[a4paper,top=2cm,bottom=2cm,left=2.5cm,right=2.5cm,marginparwidth=1.75cm]{geometry}

\usepackage{amsmath}
\usepackage{graphicx}
\usepackage[colorlinks=true, allcolors=blue]{hyperref}
\usepackage{hyperref}
\usepackage[title]{appendix}
\usepackage{mathrsfs}
\usepackage{amsfonts}
\usepackage{booktabs} 
\usepackage{caption}  
\usepackage{threeparttable} 
\usepackage{algorithm}
\usepackage{algorithmicx}
\usepackage{algpseudocode}
\usepackage{listings}
\usepackage{enumitem}
\usepackage{chngcntr}
\usepackage{booktabs}
\usepackage{lipsum}
\usepackage{subcaption}
\usepackage{authblk}
\usepackage[T1]{fontenc}    
\usepackage{csquotes}       
\usepackage{diagbox}
\usepackage{amsthm}
\usepackage{amssymb}
\usepackage{amsmath}
\usepackage{amsfonts}

\theoremstyle{definition}
\newtheorem{definition}{Definition}[section]
\theoremstyle{plain}
\newtheorem{theorem}{Theorem}[section]
\theoremstyle{plain}
\newtheorem{lemma}{Lemma}[section]
\theoremstyle{remark}
\newtheorem*{remark}{Remark}
\theoremstyle{plain}
\newtheorem{property}{Proposition}
\theoremstyle{plain}
\newtheorem{prop}{Property}
\theoremstyle{remark}
\newtheorem{ex}{Example}

\newcommand{\id}{\text{id}}
\newcommand{\yb}{\rho_{\text{YB}}}

\newcommand{\tr}{\text{Tr}}

\newcommand{\supp}{\operatorname{supp}}
\newcommand{\sinf}{\mathfrak{S}_\infty}
\newcommand{\Z}{\mathbb{Z}}

\usepackage{tikz}
\usetikzlibrary{decorations.shapes}
\usepackage[all,knot]{xy}
\newcounter{braid}
\newcounter{strands}
\pgfkeyssetvalue{/tikz/braid height}{1cm}
\pgfkeyssetvalue{/tikz/braid width}{1cm}
\pgfkeyssetvalue{/tikz/braid start}{(0,0)}
\pgfkeyssetvalue{/tikz/braid colour}{black}
\pgfkeys{/tikz/strands/.code={\setcounter{strands}{#1}}}

\makeatletter
\def\cross{%
  \@ifnextchar^{\message{Got sup}\cross@sup}{\cross@sub}}

\def\cross@sup^#1_#2{\render@cross{#2}{#1}}

\def\cross@sub_#1{\@ifnextchar^{\cross@@sub{#1}}{\render@cross{#1}{1}}}

\def\cross@@sub#1^#2{\render@cross{#1}{#2}}

\def\render@cross#1#2{
  \def\strand{#1}
  \def\crossing{#2}
  \pgfmathsetmacro{\cross@y}{-\value{braid}*\braid@h}
  \pgfmathtruncatemacro{\nextstrand}{#1+1}
  \foreach \thread in {1,...,\value{strands}}
  {
    \pgfmathsetmacro{\strand@x}{\thread * \braid@w}
    \ifnum\thread=\strand
    \pgfmathsetmacro{\over@x}{\strand * \braid@w + .5*(1 - \crossing) * \braid@w}
    \pgfmathsetmacro{\under@x}{\strand * \braid@w + .5*(1 + \crossing) * \braid@w}
    \draw[braid] \pgfkeysvalueof{/tikz/braid start} +(\under@x pt,\cross@y pt) to[out=-90,in=90] +(\over@x pt,\cross@y pt -\braid@h);
    \draw[braid] \pgfkeysvalueof{/tikz/braid start} +(\over@x pt,\cross@y pt) to[out=-90,in=90] +(\under@x pt,\cross@y pt -\braid@h);
    \else
    \ifnum\thread=\nextstrand
    \else
     \draw[braid] \pgfkeysvalueof{/tikz/braid start} ++(\strand@x pt,\cross@y pt) -- ++(0,-\braid@h);
    \fi
   \fi
  }
  \stepcounter{braid}
}

\tikzset{braid/.style={double=\pgfkeysvalueof{/tikz/braid colour},double distance=1pt,line width=2pt,white}}

\newcommand{\braid}[2][]{%
  \begingroup
  \pgfkeys{/tikz/strands=2}
  \tikzset{#1}
  \pgfkeysgetvalue{/tikz/braid width}{\braid@w}
  \pgfkeysgetvalue{/tikz/braid height}{\braid@h}
  \setcounter{braid}{0}
  \let\sigma=\cross
  #2
  \endgroup
}
\makeatother

\usepackage{xcolor}

\newcommand\ignore[1]\null

\usepackage[figurename=Fig.]{caption}
\usepackage[tablename=Tab.]{caption}

\usepackage{setspace}
\usepackage{titlesec}
\titleformat{\section} 
  {\normalfont\Large\bfseries}{\thesection.}{1em}{}

\usepackage{float}   
\usepackage{caption} 



\title{Yang-Baxter extremal characters of wreath products of finite groups with the infinite symmetric group}

\author{Hicham Assakaf}

\affil{\small Vrije Universiteit Brussel, Brussels, Belgium\\
 Ecole Normale Supérieure de Paris-Saclay, Gif-sur-Yvette, France\\
Sorbonne Université, Paris, France}
\affil{Contact: \texttt{assakaf@imj-prg.fr}}

\date{July 2024}  

\begin{document}
\maketitle

\begin{abstract}
 Let $T$ be a finite group. To a representation $\pi$ of $T$ and an involutive solution of the Yang-Baxter equation (an $R$-matrix) verifying the \textit{"extended" reflection equation}, we associate a character of a representation of the wreath product $G:=T\wr \mathfrak{S}_\infty$. The set of extremal characters of $G$ is in bijection with a continuous set of parameters. In this article, we characterize exactly what subset of parameters does correspond to an extremal Yang-Baxter character of $G$.
\end{abstract}



\section*{Introduction}
The representation theory of the infinite symmetric group, denoted by $\mathfrak{S}_{\infty}$, occupies a central place in mathematics and mathematical physics, particularly in describing the equilibrium states of quantum systems with a large number of particles. As early as the 1960s, the foundational work of E. Thoma \cite{thoma1964unzerlegbaren} classified the extremal characters of $\mathfrak{S}_{\infty}$ (the infinite-dimensional analogues of irreducible characters) using a family of continuous parameters, now known as Thoma parameters.

Recently, a profound connection has been established between this classical theory and the algebraic structures underlying quantum integrable models. In a series of works, Lechner, Pennig, and Wood \cite{lechner2019yang} demonstrated that involutive solutions to the Yang-Baxter equation ($R$-matrices) define a natural class of representations of $\mathfrak{S}_{\infty}$. In this framework, the $R$-matrix, which describes the scattering between two particles, allows for the construction of specific extremal characters, thus offering a direct physical interpretation of Thoma parameters in terms of 1+1 dimensional quantum field theory \cite{Lechner_2007}.

However, this approach is limited to systems of indistinguishable particles without complex internal structure. Yet, many physical models of interest, such as spin chains or lattice models, involve particles possessing internal degrees of freedom (spin, color, charge) \cite{faddeev} governed by a local symmetry group $T$. Mathematically, the appropriate structure to capture this combined symmetry is no longer the simple symmetric group, but the wreath product $G := T \wr \mathfrak{S}_{\infty}$. The extremal characters of this group have been studied by Hirai and Hirai \cite{hirai}.

The objective of this article is to characterize the Yang-Baxter type extremal characters for the wreath product $T \wr \mathfrak{S}_{\infty}$. Given a Hilbert space $V$, we show that to preserve the integrable structure in the presence of internal degrees of freedom, the involutive unitary $R$-matrix must satisfy, in addition to the standard Yang-Baxter equation, a compatibility condition with a representation $\pi\in End(W\otimes V)$ of the group $T$. We identify this condition as the \textbf{Extended Reflection Equation}:
\begin{equation} \label{eq:ERE}
    R_{1}\pi(t)R_{1}\pi(t^{\prime}) = \pi(t^{\prime})R_{1}\pi(t)R_{1}, \quad \forall t, t' \in T
\end{equation}
Introduced by Cherednik \cite{cherednik} to investigate quantum integrable systems on the half-line, the reflection equation serves as a boundary counterpart to the Yang–Baxter equation. While the braiding (Yang–Baxter operator) characterizes the scattering of two particles colliding in a one-dimensional system, the solution to the reflection equation describes the interaction of a particle with a boundary. This equation appears naturally in various contexts within mathematical physics, notably in the study of braid groups on higher-genus surfaces (handlebodies) introduced by Schwiebert \cite{Schwiebert94}. Here, the extended reflection equation generalizes the reflection equation and involves a more general action instead of the $K$-matrix appearing in the reflection equation.

In this work, we establish the following results:
\begin{enumerate}
    \item We define Yang-Baxter type representations for $G = T \wr \mathfrak{S}_{\infty}$ governed by a pair consisting of an $R$-matrix and a representation $\pi$ of $T$ satisfying the extended reflection equation.
    \item We recall the explicit bijection between the solutions of this system and a subset of the generalized Thoma parameters for $G$ given by \cite{hirai}.
    \item We precisely characterize this subset in the Theorem \ref{final}, thus generalizing the main theorem of \cite{lechner2019yang} to the case of internal symmetries. This characterization is the main result of this article.
\end{enumerate}

\section{Yang-Baxter representations of \texorpdfstring{$\mathfrak{S}_\infty$}{Sinf}}
The infinite symmetric group $\mathfrak{S}_\infty$, consists of all permutations of the natural numbers $\mathbb{N}$ that fix all but finitely many elements. In this part, we will study the Yang-Baxter extremal characters of this group.

\subsection{Extremal characters}
In the section we present the definition of an extremal character from \cite{borodin} and a characterization for extremality.\\
If $G$ is a finite group, and $\pi$ is a finite dimensional (complex) representation, then the \textit{character} of $\pi$ is the application $\chi^\pi:g\in G \longmapsto \tr(\pi(g))$. This definition can not be applied to more complicated groups (like the groups we will study in this paper) because when the representations become infinite dimensional, it becomes hard to define the characters using the trace. Thus, it is useful to characterize more generally the characters.
\begin{definition}
A function $\varphi:G\longrightarrow\mathbb{C}$ on a (non-necessarily finite) group $G$ is \textit{positive definite} if for any $k\geq1$, and $g_1,\dots g_k\in G$, the matrix $[\varphi(g_j^{-1}g_i)]_{1\leq i,j\leq k}$ is Hermitian positive definite. Or equivalently for any $g\in G$, we have $\varphi(g^{-1})=\overline{\varphi(g)}$, and for all $k\geq1$ and any $g_1,\dots,g_k\in G, z_1,\dots z_k\in\mathbb{C}$ we have 
$$\sum_{i,j=1}^kz_i\overline{z_j}\varphi(g_j^{-1}g_i)\geq 0$$
\end{definition}
Let $\hat{G}$ be the set of equivalency classes of irreducible representations of a finite group $G$.
\begin{property}
Let $G$ be a finite group. A function $\varphi:G\longrightarrow\mathbb{C}$ is central positive definite and takes value 1 at the unit element of $G$ if and only if it is a convex combination of normalized characters of the irreducible representations of $G$:
$$ \varphi=\sum_{\pi\in \hat{G}}c_\pi \frac{\chi^\pi}{\dim \pi},\quad c_\pi\geq0,\quad\sum_{\pi\in\hat{G}}c_\pi=1 $$
That is, all such functions form a simplex whose vertices are the normalized characters of the irreducible representations of $G$.
\end{property}

This proposition motivates the following more general definition of a character of an arbitrary group.

\begin{definition}\label{character}
\begin{itemize}
\item A \textit{character} of an arbitrary group $G$ is a function $\chi: G\longrightarrow\mathbb{C}$ which is central, positive definite, and takes value 1 at the unit element.
\item The set of all characters is a convex set. Thus we can define an \textit{extremal} character as an extremal point of this set (i.e. a point which is not a non-trivial convex combination of other characters).
\end{itemize}
\end{definition}

This generalizes the common notion of a character of a finite group. Indeed, if $G$ is a finite group, and $\pi$ is a finite-dimensional representation of $G$, then 
$$g\longmapsto \frac{\tr(\pi(g))}{\dim \pi}$$
is a character in the sense of the Definition $\ref{character}$. The notion of extremal character is also a generalization of the notion of irreducible character. Indeed, if the group is finite (or compact), then the set of all characters in sense of the Definition $\ref{character}$ is a simplex which vertices are exactly the normalized irreducible characters.

We can mention another point of view on the extremality of a character $\tau$ of a countable group $G$. 
\begin{theorem}[\cite{thoma1964unzerlegbaren}]
A character $\tau$ of a countable group $G$ is extremal if and only if $\pi_\tau(G)''$ is a factor. If $G$ is infinite, then it is necessarily a type $II_1$ factor. \\
A character $\chi$ of $\mathfrak{S}_\infty$ is extremal if and only if for all permutations $\sigma,\gamma\in\mathfrak{S}_\infty$ with disjoint supports, $$\chi(\sigma\gamma)=\chi(\sigma)\chi(\gamma)$$
\end{theorem}

\subsection{Thoma parameters}

Let $\mathbb{T}$ be the set of all sequences $(\alpha,\beta)$ with $\alpha=\{\alpha_i\}$ and $\beta=\{\beta_i\}$ verifying :
\begin{itemize}
\item $\alpha_i\geq0$, and $\beta_i\geq0$
\item $\alpha_i\geq\alpha_{i+1}$,~~$\beta_i\geq\beta_{i+1}$
\item $\sum_i\alpha_i+\beta_i\leq1$
\end{itemize}
E. Thoma proved in \cite{thoma1964unzerlegbaren} that this set of parameters is in bijection with the set of extremal characters of $\mathfrak{S}_\infty$ : to a couple $(\alpha,\beta)\in\mathbb{T}$ we associate the character verifying for all $n\geq2$ :
$$\chi(c_n)=\sum_i{\alpha_i}^n + (-1)^{n-1}{\beta_i}^n$$
with $c_n=(1~2~\cdots~n)$. As we will see in the next section, in \cite{lechner2019yang}, the authors characterized the subspace of $\mathbb{T}$ corresponding to the Yang-Baxter characters.

\subsection{Yang-Baxter characters of the infinite symmetric group}
In this section, we present the results of G. Lechner, U. Pennig and S. Wood in their article \cite{lechner2019yang} about the characterization of the Thoma parameters associated to a Yang-Baxter character of the infinite symmetric group $\mathfrak{S}_\infty$. 
Let $(~V,(\cdot,\cdot)~)$ be a finite dimensional Hilbert space of dimension $d$. \\
Let $\bigotimes_{n=1}^\infty End(V)$ to be the object defined algebraically as the inductive limit (or direct limit) of the finite tensor products.
Formally, we consider the directed system where each algebra $\mathcal{A}_n = \bigotimes_{k=1}^n \mathrm{End}(V)$ is embedded into $\mathcal{A}_{n+1}$ via the mapping $$x \mapsto x \otimes \mathrm{Id}_V.$$ 
The infinite tensor product is then identified with the union of these nested subalgebras:
\[
\bigotimes_{n=1}^{\infty} \mathrm{End}(V) := \varinjlim_{n} \mathcal{A}_n = \bigcup_{n=1}^{\infty} \mathcal{A}_n
\]
This construction ensures that any element in this space acts as the identity operator on all but a finite number of components.
\begin{definition}\label{defrepyb}
Let $R\in End(V\otimes V)$ be an involutive solution of the Yang-Baxter equation in $End(V\otimes V\otimes V)$ : $R^2=1\otimes1$ and
$$(R\otimes1)(1\otimes R)(R\otimes1)=(1\otimes R)(R\otimes1)(R\otimes1)$$
We assume $R$ to be unitary with respect to $(\cdot,\cdot)$.
We say that $R$ is an $R$-matrix.
The \textit{\textbf{Yang-Baxter representation}} $\rho_R$ of $\mathfrak{S}_\infty$ associated to the $R$-matrix $R$ is the representation defined on the generators $\sigma_i=(i~~i+1)$, for $i\geq1$ as follows : 
$$\rho_R(\sigma_i)=1^{\otimes i-1}\otimes R\otimes1\otimes\cdots\in \bigotimes_{n=1}^\infty End(V)$$
The \textbf{\textit{Yang-Baxter character}} $\chi_R$ associated is obtained by left-composing by the normalized trace $$\tau=\bigotimes_{i\geq1}\frac{\operatorname{Tr}}{d}$$
\end{definition}
It is clear that $\chi_R$ is always an extremal character, by definition of the trace and the representation. Then, we can wonder to which subset of $\mathbb{T}$ the Yang-Baxter (extremal) characters correspond. \\
Let $\mathbb{T}_{\text{YB}}\subset \mathbb{T}$ be the subset of couples $(\alpha,\beta)$ verifying :
\begin{itemize}
\item Finitely many $\alpha_i,\beta_i$ are non-zero.
\item There exists a $d\in\mathbb{N}$ such that for all $i$, $d\alpha_i,d\beta_i\in\mathbb{N}$.
\item $\sum_i \alpha_i+\beta_i=1$
\end{itemize}
\begin{theorem}[\cite{lechner2019yang}]\label{thoma}
The set $\mathbb{T}_{\text{YB}}$ corresponds to the Yang-Baxter characters of $\mathfrak{S}_\infty$ by the bijection described in the previous section.
\end{theorem}
To construct an $R$-matrix such that $\chi_R$ is associated to a given couple of Thoma parameters $(\alpha,\beta)$, the authors introduce the \textit{normal form $R$-matrices.}

\begin{definition}
     Let $V,W$ be finite dimensional Hilbert spaces and let $X\in End(V\otimes V)$, $Y\in End(W\otimes W)$. Let $F\in End((V\otimes W)\oplus(W\otimes V))$ be the flip. We define $X\boxplus Y\in End((V\oplus W)\otimes (V\oplus W))$ as
	\begin{align*}\label{eq:DefBoxplus}
	       X\boxplus Y &= X\oplus Y\oplus F\quad\text{on}\\
	       (V\oplus W)\otimes (V\oplus W) &= (V\otimes V)\oplus(W\otimes W)\oplus((V\otimes W)\oplus(W\otimes V))
	\end{align*}	
\end{definition}

In other words, $X\boxplus Y$ acts as $X$ on $V\otimes V$, as $Y$ on $W\otimes W$, and as the flip on the mixed tensors involving factors from both, $V$ and $W$. Note that the above definition works in the same way for infinite dimensional Hilbert spaces.

Let $(\alpha,\beta)\in\mathbb{T}_{\text{YB}}$. Let $d_i^+:=d\alpha_i\in\mathbb{N}$ and $d_i^-:=d\beta_i\in\mathbb{N}$. We introduce a decomposition  $$V=\bigoplus_{i,\varepsilon} V_{d_i}^\varepsilon$$
such that $\text{dim}(V_{d_i}^\varepsilon)=d_i^\varepsilon$ and let $$
N:=\left(\boxplus_{i}1_{d_i^+}\right)\boxplus\left(\boxplus_{i} (-1_{d_i^-})\right)\in End(V\otimes V)$$
where $N$ acts as $\varepsilon\cdot id$ on $V_{d_i}^\varepsilon\otimes V_{d_i}^\varepsilon$ and as the flip on $V_{d_i}^\varepsilon\otimes V_{d_j}^{\varepsilon'}$ with $(i,\varepsilon)\neq (j,\varepsilon')$. 

\begin{prop}[\cite{lechner2019yang} Proposition 4.3 - 4.6]
$N$ is an $R$-matrix and the Thoma parameters associated to $N$ are $(\alpha,\beta)$.
\end{prop}

More generally, the construction gives the following property
\begin{prop}[\cite{lechner2019yang}]
Let $V,W$ be two Hilbert spaces of dimension $d,d'$. If $R,R'$ are $R$-matrices on $V\otimes V$ and $W\otimes W$, associated to Thoma parameters $(\alpha,\beta)$ and $(\tilde{\alpha},\tilde{\beta})$, then $R\boxplus R'$ is an $R$-matrix with Thoma parameters $(\widehat{\alpha},\widehat{\beta})$ such that 
$$\{\widehat{\alpha}_i\}=\{\frac{d}{d+d'}\alpha_k,\frac{d'}{d+d'}\tilde{\alpha}_\ell;k,\ell\}$$
$$\{\widehat{\beta}_i\}=\{\frac{d}{d+d'}\beta_k,\frac{d'}{d+d'}\tilde{\beta}_\ell;k,\ell\}$$
\end{prop}

\section{Yang-Baxter extremal characters of \texorpdfstring{$T\wr\mathfrak{S}_\infty$}{TSinf}}
Let $ T $ be a finite group. The infinite symmetric group $ \mathfrak{S}_\infty $ acts naturally on the set $\bigcup_{n} T^n$ by permuting the coordinates of the elements. This action allows us to define the semi-direct product $\left(\bigcup_{n} T^n \right) \rtimes \mathfrak{S}_\infty $. We denote this semi-direct product by $ G := T \wr \mathfrak{S}_\infty $, where $ T \wr \mathfrak{S}_\infty $ is the wreath product of $T$ with $ \mathfrak{S}_\infty $.

\subsection{Yang-Baxter representations of the group \texorpdfstring{$G$}{G}}
Let $V,W$ be two finite dimensional Hilbert spaces. To generalize the previous result to the extremal characters of $G$, we firstly have to define what a Yang-Baxter character is. To define something canonical, we shall proceed as follows 
\begin{definition}
    Let $\pi$ be a unitary representation of $T$ on $W\otimes V$, and $R\in End(V\otimes V)$ an involutive unitary $R$-matrix.\\
    Let $\rho_R:\sinf\longrightarrow\bigotimes_{n=0}^\infty End(V)$ to be the Yang-Baxter representation of $\sinf$ associated to $R$ defined in Definition \ref{defrepyb}.
    For all $g=(d,\sigma)\in G$, let $$\rho_{\pi,R}(g):=\rho_{\pi,R}((d,1))\rho_{\pi,R}((1,\sigma))\in End(W)\bigotimes End(V)^{\otimes\infty},$$
    \noindent with $$\rho_{\pi,R}((1,\sigma)):=1\otimes\rho_R(\sigma)\in End(W)\bigotimes End(V)^{\otimes\infty},$$ 
    and if $ d=(t_i)_{i\in\mathbb{N}}\in\bigcup_{n}T^n$ then
    $$\rho_{\pi,R}((d,1)):=\widetilde\pi(t_1)\times\left(R_1\widetilde\pi(t_2)R_1\right)\times\cdots\times\left(R_{n-1}\dots R_1\widetilde\pi(t_n)R_1\dots R_{n-1}\right)\times\cdots,$$
   where $$R_i:=1\otimes 1^{\otimes (i-1)}\otimes R\otimes 1^{\otimes\infty}\in End(W)\bigotimes End(V)^{\otimes\infty},$$
   and $$\widetilde\pi(t):=\pi(t)\otimes1^{\otimes\infty}\in End(W)\bigotimes End(V)^{\otimes\infty} .$$
   We will write $\pi$ instead of $\widetilde\pi$ in all the paper.
     Again $\rho_{\pi,R}((d,1)\in End(W)\bigotimes End(V)^{\otimes\infty} $.\\
    We say that $(\pi,R)$ is a \textit{Yang-Baxter couple} if \begin{equation*}
    \rho_{\pi,R}\colon\left\{
    \begin{aligned}
        G&\longrightarrow End(W)\bigotimes End(V)^{\otimes\infty}\\
        g&\longmapsto \rho_{\pi,R}(g)
    \end{aligned}
    \right.
\end{equation*}
defines a representation of $G$. We will call it $\rho$ if there is no ambiguity on $(\pi,R)$.
\end{definition}
\begin{ex}
    If $1$ denotes the trivial representation and $R$ is an involutive unitary $R$-matrix then $(1,R)$ is a Yang-Baxter couple. When we compose $\rho_{1,R}$ by $\tau$ we obtain the extremal Yang-Baxter associated to $R$.
\end{ex}

\begin{property} \label{extendedRE}
    A couple $(\pi,R)$ is a Yang-Baxter couple if and only if $R$ verify the Yang-Baxter equation and $(\pi,R)$ verify the \textbf{extended reflection equation} :
    $$\forall t,t'\in T,\qquad R_1\pi(t)R_1\pi(t')=\pi(t')R_1\pi(t)R_1,$$
    where we write $\pi(s)$ instead of $\pi(s)\otimes1$
\end{property}

\begin{remark}
If $T=\Z / 2 \Z$, then $G=\Z / 2 \Z\wr\sinf$ is the Weyl group of type $B_\infty$. Defining a unitary representation $\pi$ of $T$ on the vector space $W\otimes V$ corresponds to choosing an involutive unitary matrix $K$. Moreover, the couple $(\pi,R)$ satisfies the \textit{extended reflection equation} if and only if $K$ and $R$ verify the reflection equation :
$$ (K\otimes 1)(1\otimes R)(K\otimes 1)(1\otimes R) = (1\otimes R)(K\otimes 1)(1\otimes R)(K\otimes 1) $$
That is the reason why we call it \textit{extended} reflection equation.
\end{remark}

To simplify the proof of this property, we will introduce a geometric way to visualize the operations in the group $G$.

\subsubsection{Geometric representation}
We already know that the Yang-Baxter representation of the symmetric group can be geometrically represented with braids. The Yang-Baxter equation is written as follows :

\vskip 7pt
\begin{center}
\begin{tikzpicture}
\braid[strands=3,braid start={(0,0)}]%
{\sigma_1^{-1} \sigma_2^{-1} \sigma_1^{-1}}
\node[font=\Huge] at (4.5,-1.5) {\(=\)};
\braid[strands=3,braid start={(5,0)}]
{\sigma_2^{-1} \sigma_1^{-1} \sigma_2^{-1}}
\node[font=] at (2,-3.5) {$R_1R_2R_1$};
\node[font=] at (7,-3.5) {$R_2R_1R_2$};
\end{tikzpicture}
\end{center}

\vskip 7pt
To represent the group $G$ with braids, we add a red vertical strand on the left, and for all $t\in T$, we represent the element $\pi(t)$ as follows :

\vskip 7pt 
\begin{center}
\begin{tikzpicture}
\draw[white,double=red, very thick,-] (-1,-1.5) -- (-1,0);
\draw[smooth,white,double=black,line width=1mm,-] plot[variable=\x,domain=-1.5:1.5] ({-1.3*exp(-3.5*\x*\x)},{\x});
\draw[white,double=red, very thick,-] (-1,0) -- (-1,1.5);
\foreach \x in {1,2,3}{
\draw[white,double=black, thick,-] (\x,-1.5) -- (\x,1.5);
}
\node[anchor=south] at (-1.7,-0.3) {$\pi(t)$};
\end{tikzpicture}
\end{center}

Then, the extended reflection equation can be represented by the following statement : for all $t,t'\in T$,
\vskip 0pt
\begin{center}
\scalebox{0.7}{%
\begin{tikzpicture}
\draw[white,double=red, thick,-] (-1,-1.5) -- (-1,0);
\draw[smooth,white,double=black,line width=1mm,-] plot[variable=\x,domain=-1.5:1.5] ({-1.3*exp(-3.5*\x*\x)},{\x});
\draw[white,double=red, thick,-] (-1,0) -- (-1,1.5);
\draw[white,double=black, thick,-] (1,-1.5) -- (1,1.5);
\node[anchor=south] at (-1.8,-0.3) {$\pi(t)$};

\draw[smooth,white,double=black,line width=1mm,-] plot[variable=\x,domain=-2.5:-1.5] ({-1*exp(-10*(\x+1.5)*(\x+1.5))+1},{\x});
\draw[smooth,white,double=black,line width=1mm,-] plot[variable=\x,domain=-2.5:-1.5] ({exp(-10*(\x+1.5)*(\x+1.5))},{\x});
\draw[white,double=red, thick,-] (-1,-2.5) -- (-1,-1.5);
\node[anchor=south] at (-1.8,-4.2) {$\pi(t')$};

\draw[white,double=red, thick,-] (-1,-5.5) -- (-1,-4);
\draw[smooth,white,double=black,line width=1mm,-] plot[variable=\x,domain=-5.5:-2.5] ({-1.3*exp(-3.5*(\x+4)*(\x+4))},{\x});
\draw[white,double=red, thick,-] (-1,-4) -- (-1,-2.5);
\draw[white,double=black, thick,-] (1,-5.5) -- (1,-2.5);
\node[anchor=south] at (3.2,-1.3) {$\pi(t')$};

\draw[smooth,white,double=black,line width=1mm,-] plot[variable=\x,domain=-6.5:-5.5] ({-1*exp(-10*(\x+5.5)*(\x+5.5))+1},{\x});
\draw[smooth,white,double=black,line width=1mm,-] plot[variable=\x,domain=-6.5:-5.5] ({exp(-10*(\x+5.5)*(\x+5.5))},{\x});
\draw[white,double=red, thick,-] (-1,-6.5) -- (-1,-5.5);
\node[font=] at (0,-7) {$\pi(t)R_1\pi(t')R_1$ };
\node[font=] at (5,-7) {$R_1\pi(t')R_1\pi(t)$ };
\node[anchor=south] at (3.2,-5.2) {$\pi(t)$};
\node[font=\Huge] at (2.5,-2.5) {\(=\)};

\draw[smooth,white,double=black,line width=1mm,-] plot[variable=\x,domain=0.5:1.5] ({-1*exp(-10*(\x-1.5)*(\x-1.5))+6},{\x});
\draw[smooth,white,double=black,line width=1mm,-] plot[variable=\x,domain=0.5:1.5] ({exp(-10*(\x-1.5)*(\x-1.5))+5},{\x});
\draw[white,double=red, thick,-] (4,0.5) -- (4,1.5);

\draw[white,double=red, thick,-] (4,-2.5) -- (4,-1);
\draw[smooth,white,double=black,line width=1mm,-] plot[variable=\x,domain=-2.5:0.5] ({-1.3*exp(-3.5*(\x+1)*(\x+1))+5},{\x});
\draw[white,double=red, thick,-] (4,-1) -- (4,0.5);
\draw[white,double=black, thick,-] (6,-2.5) -- (6,0.5);

\draw[smooth,white,double=black,line width=1mm,-] plot[variable=\x,domain=-3.5:-2.5] ({-1*exp(-10*(\x+2.5)*(\x+2.5))+6},{\x});
\draw[smooth,white,double=black,line width=1mm,-] plot[variable=\x,domain=-3.5:-2.5] ({exp(-10*(\x+2.5)*(\x+2.5))+5},{\x});
\draw[white,double=red, thick,-] (4,-3.5) -- (4,-2.5);

\draw[white,double=red, thick,-] (4,-6.5) -- (4,-5);
\draw[smooth,white,double=black,line width=1mm,-] plot[variable=\x,domain=-6.5:-3.5] ({-1.3*exp(-3.5*(\x+5)*(\x+5))+5},{\x});
\draw[white,double=red, thick,-] (4,-5) -- (4,-3.5);
\draw[white,double=black, thick,-] (6,-6.5) -- (6,-3.5);
\end{tikzpicture}}
\end{center}

\begin{remark}
The fact that all $\sigma_i$ are involutive is equivalent to the equality :
\begin{center}
\begin{tikzpicture}
\braid[strands=2,braid start={(0,0)}]%
{\sigma_1 }
\node[font=\Huge] at (4,-0.5) {\(=\)};
\braid[strands=2,braid start={(5,0)}]
{\sigma_1^{-1} }
\end{tikzpicture}
\end{center}
\end{remark}

Thus, we can now write the proof of the Proposition \ref{extendedRE}.

\begin{proof}[Proof of the Proposition \ref{extendedRE}]
    We assume $(\pi,R)$ is a Yang-Baxter couple. Let $\widetilde\sinf=\{1\}\times\sinf\subset G$. The subgroup $\widetilde\sinf$ is isomorphic to $\sinf$. The subrepresentation $\rho_{\pi,R| \widetilde\sinf}=1\otimes\rho_R$ and $\rho_R$ is the Yang-Baxter representation associated to $R$. Thus $R$ satisfies the Yang-Baxter relation. Moreover, let $t,t'\in T$, let $d=(t,1,1,\dots)$ and $d'=(1,t',1,1,\dots)$. Let $g=(d,\id)\in G$ and $g'=(d',\id)\in G$. We have $gg'=g'g$, then $\rho_{\pi,R}(g)\rho_{\pi,R}(g')=\rho_{\pi,R}(g')\rho_{\pi,R}(g)$. Thus, $\pi(t)R_1\pi(t')R_1=R_1\pi(t')R_1\pi(t)$, and the conditions are necessary.
    Let us prove now that they are sufficient. 
    Let $T=\left\langle X|R\right\rangle$ be a presentation of $T$. We obtain the following presentation :
    $$H:=\bigcup_{n\geq1}T^n=\left\langle (X_n)_{n\geq1}|(R_n)_{n\geq1},~~xy=yx,\quad\forall(x,y)\in X_k\times X_\ell,~~k\neq\ell\right\rangle$$
    where for all $n$, $(X_n,R_n)$ is a copy of $(X,R)$.
    $\pi$ being a representation of $T$, the application 
    \begin{equation*}
    \pi\colon\left\{
    \begin{aligned}
        T&\longrightarrow Gl(W\otimes V)\subset Gl(W\otimes V^{\otimes\infty})\\
        g&\longmapsto \pi(g)
    \end{aligned}
    \right.
\end{equation*}
is a group homomorphism. Then if $H'$ is the group with the same presentation as $H$ but without the commutation condition, and $K=\left\langle[X_n,X_k],k\neq n\right\rangle$ such that $H\simeq H'/K$ we define $\rho$ as the following homomorphism :
\begin{equation*}
    \rho\colon\left\{
    \begin{aligned}
        H'&\longrightarrow Gl(W\otimes V)\subset Gl(W\otimes V^{\otimes\infty})\\
        g=(t_i)_i&\longmapsto \prod_{i\in supp(g)} (R_{i-1}\cdots R_1 \pi(t_i) R_1\cdots R_{i-1})
    \end{aligned}
    \right.
\end{equation*}
We shall now prove that $K\subset\ker \rho$. Then by factorizing, we would obtain a group homomorphism $\rho:H\longrightarrow Gl(W\otimes V^{\otimes\infty})$ (we will still call it $\rho$).

\noindent Let $k<n\in\mathbb{N}$. We have to prove that for all $(x,y)\in X_k\times X_n$, $\rho(xy)=\rho(yx)$. By definition of $\rho$ and the $X_i$, it is equivalent to prove that
for all $t,t'\in T$,
\begin{align*}
\left(R_{n-1}\cdots R_1\pi(t)R_1\cdots R_{n-1}\right)&\left(R_{k-1}\cdots R_1\pi(t')R_1\cdots R_{k-1}\right)=\\[3mm]
(R_{k-1}\cdots R_1&\pi(t')R_1\cdots R_{k-1})\left(R_{n-1}\cdots R_1\pi(t)R_1\cdots R_{n-1}\right)    
\end{align*}
But $(\pi,R)$ verifying the Yang-Baxter equation and the extended reflection equation, so we can compute the product in the left hand side to obtain the one in the right hand side. This equality is more obvious with the geometric representation of the operations on the group. Indeed, the left hand side is represented by the following braids : 
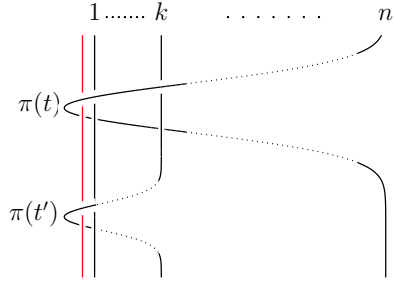
\begin{center}
\scalebox{0.8}{%
\tikzset{decorate sep/.style 2 args=
{decorate,decoration={shape backgrounds,shape=circle,shape size=#1,shape sep=#2}}}
\begin{tikzpicture}
\draw[white,double=red,very thick,-] (-5,2) -- (-5,-1);
\foreach \x in {-4.8}{
\draw[white,double=black,very thick,-] (\x,2) -- (\x,-1);
}
\draw[smooth,white,double=black,line width=1mm,-] plot[variable=\x,domain=0:0.4] ({-1.6*exp(-10*(\x-1)*(\x-1))-3.7},{-\x});
\draw[smooth,white,double=black,line width=1mm,-] plot[variable=\x,domain=0.8:1.2] ({-1.6*exp(-10*(\x-1)*(\x-1))-3.7},{-\x});
\draw[smooth,white,double=black,line width=1mm,-] plot[variable=\x,domain=1.6:2] ({-1.6*exp(-10*(\x-1)*(\x-1))-3.7},{-\x});

\draw[smooth,white,double=black, dotted] plot[variable=\x,domain=1.2:1.6] ({-1.6*exp(-10*(\x-1)*(\x-1))-3.7},{-\x});
\draw[smooth,white,double=black, dotted] plot[variable=\x,domain=0.4:0.8] ({-1.6*exp(-10*(\x-1)*(\x-1))-3.7},{-\x});
\node[anchor=south] at (-5.7,0.5) {$\pi(t)$};
\node[anchor=south] at (-5.8,-1.3) {$\pi(t')$};
\node[anchor=south] at (0,2.1) {$n$};
\foreach \x in {-4.8}{
\draw[white,double=black,very thick,-] (\x,-1) -- (\x,-2);
}
\draw[white,double=red,very thick,-] (-5,-1) -- (-5,-2);
\node[anchor=south] at (-4.8,2.1) {$1$};
\node[anchor=south] at (-3.7,2.1) {$k$};
\draw[decorate sep={0.2mm}{1mm},fill] (-4.6,2.3) -- (-4,2.3);
\draw[decorate sep={0.2mm}{2.5mm},fill] (-2.6,2.3) -- (-1,2.3);
\draw[white,double=black,very thick,-] (-3.7,2) -- (-3.7,0);

\draw[smooth,white,double=black,line width=1mm,-] plot[variable=\x,domain=-2:-1.7] ({-5.3*exp(-3*(\x+0.8)*(\x+0.8))},{-\x});
\draw[smooth,white,double=black,line width=1mm,-] plot[variable=\x,domain=-1.2:-0.4] ({-5.3*exp(-3*(\x+0.8)*(\x+0.8))},{-\x});
\draw[smooth,white,double=black,line width=1mm,-] plot[variable=\x,domain=0.1:2] ({-5.3*exp(-3*(\x+0.8)*(\x+0.8))},{-\x});
\draw[smooth,white,double=black, dotted] plot[variable=\x,domain=-1.7:-1.1] ({-5.3*exp(-3*(\x+0.8)*(\x+0.8))},{-\x});
\draw[smooth,white,double=black, dotted] plot[variable=\x,domain=-0.5:0.1] ({-5.3*exp(-3*(\x+0.8)*(\x+0.8))},{-\x});

\draw[white,double=black,very thick,-] (-3.7,0.8) -- (-3.7,0.1);
\draw[white,double=black,very thick,-] (-4.8,0.8) -- (-4.8,0.5);
\draw[white,double=red,very thick,-] (-5,0.8) -- (-5,0.5);

\end{tikzpicture}}
\end{center}

\noindent And the right hand side is represented by the following braids.

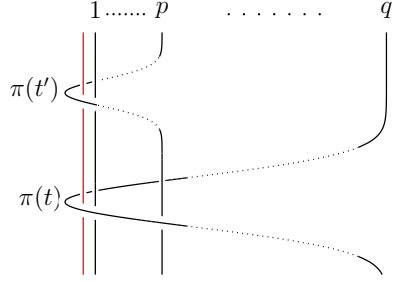
\begin{center}
\scalebox{0.8}{%
\tikzset{decorate sep/.style 2 args=
{decorate,decoration={shape backgrounds,shape=circle,shape size=#1,shape sep=#2}}}
\begin{tikzpicture}
\draw[white,double=red,very thick,-] (-5,-2) -- (-5,1);
\foreach \x in {-4.8}{
\draw[white,double=black,very thick,-] (\x,-2) -- (\x,1);
}
\draw[smooth,white,double=black,line width=1mm,-] plot[variable=\x,domain=0:0.4] ({-1.6*exp(-10*(\x-1)*(\x-1))-3.7},{\x});
\draw[smooth,white,double=black,line width=1mm,-] plot[variable=\x,domain=0.8:1.2] ({-1.6*exp(-10*(\x-1)*(\x-1))-3.7},{\x});
\draw[smooth,white,double=black,line width=1mm,-] plot[variable=\x,domain=1.6:2] ({-1.6*exp(-10*(\x-1)*(\x-1))-3.7},{\x});

\node[anchor=south] at (-5.7,-1.1) {$\pi(t)$};
\node[anchor=south] at (-5.8,0.7) {$\pi(t')$};
\draw[smooth,white,double=black, dotted] plot[variable=\x,domain=1.2:1.6] ({-1.6*exp(-10*(\x-1)*(\x-1))-3.7},{\x});
\draw[smooth,white,double=black, dotted] plot[variable=\x,domain=0.4:0.8] ({-1.6*exp(-10*(\x-1)*(\x-1))-3.7},{\x});

\node[anchor=south] at (0,2.1) {$q$};
\foreach \x in {-4.8}{
\draw[white,double=black,very thick,-] (\x,1) -- (\x,2);
}
\draw[white,double=red,very thick,-] (-5,1) -- (-5,2);
\node[anchor=south] at (-4.8,2.1) {$1$};
\node[anchor=south] at (-3.7,2.1) {$p$};
\draw[decorate sep={0.2mm}{1mm},fill] (-4.6,2.3) -- (-4,2.3);
\draw[decorate sep={0.2mm}{2.5mm},fill] (-2.6,2.3) -- (-1,2.3);
\draw[white,double=black,very thick,-] (-3.7,-2) -- (-3.7,0);

\draw[smooth,white,double=black,line width=1mm,-] plot[variable=\x,domain=-2:-1.7] ({-5.3*exp(-3*(\x+0.8)*(\x+0.8))},{\x});
\draw[smooth,white,double=black,line width=1mm,-] plot[variable=\x,domain=-1.2:-0.4] ({-5.3*exp(-3*(\x+0.8)*(\x+0.8))},{\x});
\draw[smooth,white,double=black,line width=1mm,-] plot[variable=\x,domain=0.1:2] ({-5.3*exp(-3*(\x+0.8)*(\x+0.8))},{\x});
\draw[smooth,white,double=black, dotted] plot[variable=\x,domain=-1.7:-1.1] ({-5.3*exp(-3*(\x+0.8)*(\x+0.8))},{\x});
\draw[smooth,white,double=black, dotted] plot[variable=\x,domain=-0.5:0.1] ({-5.3*exp(-3*(\x+0.8)*(\x+0.8))},{\x});

\draw[white,double=black,very thick,-] (-3.7,-0.8) -- (-3.7,-0.1);
\draw[white,double=black,very thick,-] (-4.8,-0.8) -- (-4.8,-0.5);
\draw[white,double=red,very thick,-] (-5,-0.8) -- (-5,-0.5);

\end{tikzpicture}}
\end{center}

So hence the commutativity and $\rho:H\longrightarrow Gl(W\otimes V^{\otimes\infty})$ is a group homomorphism. 
Then $G=H\rtimes\mathfrak{S}_\infty$ and if we have $$\mathfrak{S}_\infty=\left\langle (\sigma_i)_{i\geq1}|\sigma_i\sigma_{i+1}\sigma_i=\sigma_{i+1}\sigma_i\sigma_{i+1},~~ \sigma_i\sigma_j=\sigma_j\sigma_i\quad \forall |i-j|>1\right\rangle=\left\langle Y|S\right\rangle$$
with $Y=(\sigma_i)_i$ and $S$ the relations they have to verify, then 
$$G=\left\langle (X_n)_n,Y|(R_n)_n,S,K,~~ \sigma d \sigma^{-1}=\sigma\cdot d,\quad \forall\sigma\in Y \right\rangle$$
Let $G'$ be the group with the same presentation as $G$ but without the "action relations" (it is basically the free product $H*\mathfrak{S}_\infty$). 

\noindent Let $\rho':G'\longrightarrow Gl(W\otimes V^{\otimes\infty})$ be the homomorphism such that $\rho'_{|H}=\rho$ and $\rho'_{|\mathfrak{S}_\infty}$ is the classical Yang-Baxter representation sending $\sigma_i$ on $1^{\otimes i-1}\otimes R\otimes1\otimes \cdots$. We will call it $\rho$.

\noindent If $L$ is the subgroup generated by the "action relations" it means that $G\simeq G'/L$. To finally obtain the representation of $G$, we shall prove that $L\subset \ker \rho$. This is obvious by construction of $\rho_{|H}$. Indeed, if $d=(1,\dots,1,t)\in T^n$ then if $n<i$ :
$$R_i\rho(d)R_i=R_i\times R_{n-1}\cdots R_1\pi(t)R_1\cdots R_{n-1}\times R_i=\rho(d)=\rho(\sigma_{i}\cdot d)$$
\noindent if $n=i$,
$$R_i\rho(d)R_i=R_i\times R_{i-1}\cdots R_1\pi(t)R_1\cdots R_{i-1}\times R_i=\rho(\sigma_{i}\cdot d)$$
\noindent if $n=i+1$,
$$R_i\rho(d)R_i=R_i\times R_{i}\cdots R_1\pi(t)R_1\cdots R_{i}\times R_i=\rho(\sigma_{i}\cdot d)$$
\noindent if $n\geq i+2$,
$$R_i\rho(d)R_i=R_{n-1}\cdots R_iR_{i+1}R_i\cdots R_1\pi(t)R_1\cdots R_iR_{i+1}R_i\cdots R_{n-1}=\rho(\sigma_{i}\cdot d)$$
Thus, by the Yang-Baxter equation, and because $R_{i+1}$ commutes with $R_j$ for $j\leq i-1$,
$$R_i\rho(d)R_i=R_{n-1}\cdots R_{i+1}R_i\cdots R_1\pi(t)R_1\cdots R_iR_{i+1}\cdots R_{n-1}=\rho(d)=\rho(\sigma_{i}\cdot d)$$
Finally, $L\subset \ker\rho$ and again by factorizing, we can define a homomorphism called $\rho_{\pi,R}:G\longrightarrow Gl(W\otimes V^{\otimes\infty})$. 
\end{proof}

Thus, this construction is canonical and we have defined what is a Yang-Baxter representation for any wreath product $T\wr\mathfrak{S}_\infty$. 
\begin{remark}
    The construction above also stands for a non necessarily finite group $T$.
\end{remark}
  
\subsection{Extremal Yang-Baxter characters of \texorpdfstring{$G$}{G}}
\subsubsection{Conjugacy classes in \texorpdfstring{$G=T\wr\mathfrak{S}_\infty$}{G=TSinf}}
In this section, we present the study of the conjugacy classes of elements of $G$ in the section 2.1 in \cite{hirai}.

Again, we have a standard decomposition for all elements of $G$. 
\begin{definition}
Let $g=(d,\sigma)\in G$ with $d=(d_i)_{i\in\mathbb{N}}\in T^\mathbb{N}$ such that for $n$ large enough, $d_n=e_T$, and $\sigma\in\mathfrak{S}_\infty$. The \textit{support} of $d$ is defined as the set
$supp(d):=\{i\in\mathbb{N}~|~d_i\neq e_T\}$.
The \textit{support} of $\sigma$ $supp(\sigma)$ denotes as usual the set $\{i\in\mathbb{N}~|~\sigma(i)\neq i\}$. The \underline{\textit{support}} of $g$ is defined as follows :
$$supp(g)= supp(d)\cup supp(\sigma)$$
We say that $g=(d,\sigma)$ is \underline{\textit{cyclic}} if $\sigma$ is cyclic and $supp(d)\subset supp(\sigma)$. If $\sigma$ is cyclic, the \underline{\textit{length}} of $\sigma$ $\ell(\sigma)$ is the cardinal of $\supp(\sigma)$. We say that $\xi=(d,\sigma)$ is \underline{\textit{elementary}} if $\sigma=id$ and $supp(d)=\{q\}$ for a $q\in\mathbb{N}$.
\end{definition}

\begin{property}[The standard decomposition]\label{standard}
For an arbitrary element $g=(d,\gamma)\in G$, there exists a decomposition as a product of cyclic elements $g_j=(d_j,\gamma_j)\in G$ and elementary elements $\xi_{q_i}=(d'_i,id)\in G$ with $supp(d'_i)=\{q_i\}$, unique up to the order of the elements such that the supports of the $g_j$ and $\xi_{q_i}$ are disjoint :
$$g=\xi_{q_1}\xi_{q_2}\dots\xi_{q_r}g_1g_2\dots g_s$$
Furthermore, $\gamma=\gamma_1\gamma_2\dots\gamma_s$ is the disjoint cyclic decomposition of $\gamma$. 
\end{property}
We denote by $[t]$ the conjugacy class of $t \in T$, and by $T / \sim$ the set of all conjugacy classes of $T$, and $t \sim t'$ denotes that $t, t' \in T$ are conjugate in $T$. For $g_j=\left(d_j, \sigma_j\right)\in G$ where $\sigma_j$ is a cyclic permutation and $\supp(d_j)\subset\supp\sigma_j$, let $\sigma_j=\left(\begin{array}{llll}i_{j, 1} & i_{j, 2} & \ldots & i_{j, \ell_j}\end{array}\right)$ and put $K_j:=\supp\left(\sigma_j\right)=\left\{i_{j, 1}, i_{j, 2}, \ldots, i_{j, \ell_j}\right\}$ with $\ell_j=\ell\left(\sigma_j\right)$. For $d_j=\left(t_i\right)_{i \in K_j}$, we put
$$
P_{\sigma_j}\left(d_j\right):=\left[t_{\ell_j}^{\prime} t_{\ell_j-1}^{\prime} \cdots t_2^{\prime} t_1^{\prime}\right] \in T / \sim \quad \text { with } \quad t_k^{\prime}=t_{i_{j, k}} \quad\left(1 \leq k \leq \ell_j\right) .
$$

\noindent We recall that for $g=(d,\sigma)\in G$, $\supp(g)=\supp(d)\cup\supp(\sigma)$.

\begin{lemma} 
\begin{itemize}
\item[(i).] Let $\sigma \in \mathfrak{S}_{\infty}$ be a cycle, and put $K=\operatorname{supp}(\sigma)$. Then, an element $g=(d, \sigma) \in G$ with $\operatorname{supp}(d)\subset K$ is conjugate in it to $g^{\prime}=\left(d^{\prime}, \sigma\right) \in$ $G$ with $d^{\prime}=\left(t_i^{\prime}\right)_{i \in K}, t_i^{\prime}=e_T\left(i \neq i_0\right),\left[t_{i_0}^{\prime}\right]=P_\sigma(d)$ for some $i_0 \in K$.
\item[(ii).] Identify $\tau \in \mathfrak{S}_{\infty}$ with its image in $G=T\wr\mathfrak{S}_{\infty}$. Then we have, for $g=(d, \sigma)$,
$$
\tau g \tau^{-1}=\left(\tau(d), \tau \sigma \tau^{-1}\right)=:\left(d^{\prime}, \sigma^{\prime}\right),
$$
\noindent and $P_{\sigma^{\prime}}\left(d^{\prime}\right)=P_\sigma(d)$.
\end{itemize}
\end{lemma}
\begin{proof} 
\begin{itemize}
\item[(i).] We may assume that $\sigma=\left(\begin{array}{llll}1 & 2 & \cdots & \ell\end{array}\right)$ and so $K=\boldsymbol{I}_{\ell}=$ $\{1,2, \ldots, \ell\}$. Then, for $s=\left(s_1, s_2, \ldots, s_{\ell}\right) \in T^\ell \hookrightarrow G$, we have $s g s^{-1}=\left(d^{\prime \prime}, \sigma\right)$ with $d^{\prime \prime}=\left(t_i^{\prime \prime}\right)_{i \in K}$,
$$
t_i^{\prime \prime}=s_i t_i\left(s_{\sigma^{-1}(i)}\right)^{-1}=s_i t_i\left(s_{i-1}\right)^{-1} \quad(1 \leq i \leq \ell, 0 \equiv \ell) .
$$

\noindent Therefore $t_{\ell}^{\prime \prime} t_{\ell-1}^{\prime \prime} \cdots t_2^{\prime \prime} t_1^{\prime \prime}=s_{\ell}\left(t_{\ell} t_{\ell-1} \cdots t_2 t_1\right) s_{\ell}^{-1}$, and so $P_\sigma\left(d^{\prime \prime}\right)=P_\sigma(d)$.
Take $s_{\ell}=e_T, s_1=t_1^{-1}, s_2=\left(t_2 t_1\right)^{-1}, \ldots, s_{\ell-1}=\left(t_{\ell-1} \cdots t_2 t_1\right)^{-1}$, then we get $t_i^{\prime \prime}=e_T(1 \leq i<\ell)$ and $t_{\ell}^{\prime \prime}=t_{\ell} t_{\ell-1} \cdots t_2 t_1$.
\item[(ii).] With $\sigma$ above, 
we have $\tau \sigma \tau^{-1}=(\tau(1) \quad \tau(2) \quad \cdots \quad \tau(\ell))$, 
and $d^{\prime}=$
$\tau(d)=\left(t_j^{\prime}\right)_{j \in K^{\prime}}, K^{\prime}=\tau(K)$, 
with $t_j^{\prime}=t_{\tau^{-1}(j)}$ and so $t_{\tau(i)}^{\prime}=t_i(i \in K)$. 
Hence $t_{\tau(\ell)}^{\prime} t_{\tau(\ell-1)}^{\prime} \cdots t_{\tau(2)}^{\prime} 
t_{\tau(1)}^{\prime}=t_{\ell} t_{\ell-1} \cdots t_2 t_1$. \\
This proves the assertion.
\end{itemize}
\end{proof}

Applying this lemma to each component $g_j=\left(d_j, \sigma_j\right)$ in the standard decomposition of $g \in G$, we get the following characterization of the conjugacy classes of $G$.

\begin{theorem} Let $T$ be a finite group. Take an element $g \in G=T\wr\sinf$ and let its standard decomposition be
$$
g=\xi_{q_1} \xi_{q_2} \cdots 
\xi_{q_r} g_1 g_2 \cdots g_m
$$
\noindent with $\xi_{q_k}=\left(t_{q_k},1\right)$, and $g_j=\left(d_j, \sigma_j\right)$, $\sigma_j$ cyclic, $\supp\left(d_j\right) \subset$ $\supp\left(\sigma_j\right)$. \\
Then the conjugacy class of $g$ is determined by
$\left[t_{q_k}\right] \in T / \sim(1 \leq k \leq r)$ and $\left(P_{\sigma_j}\left(d_j\right), \ell\left(\sigma_j\right)\right)(1 \leq j \leq m)$,
where $P_{\sigma_j}\left(d_j\right) \in T / \sim$ and $\ell\left(\sigma_j\right) \geq 2$.
\end{theorem}

\subsubsection{Characterization of extremal Yang-Baxter characters of \texorpdfstring{$G$}{G}}

\begin{theorem}[Theorem 12, \cite{hirai}]
A character $\chi$ of $G$ is extremal if and only if for all $g,g'\in G$ with disjoint supports, 
$$\chi(gg')=\chi(g)\chi(g')$$
\end{theorem}

Let $T$ be a finite group. Let $\hat{T}$ be the set of the equivalence classes of its irreducible characters. We identify the equivalence classes with one of their representative. In the article \cite{hirai}, the authors introduce the following set $\mathbb{P}$ of families of parameters :
$$a_{\zeta,\varepsilon}=(a_{\zeta,\varepsilon,i})_{i\in\mathbb{N}}\quad \text{and}\quad \mu_\zeta,\quad\text{for all }\zeta\in\hat{T},\varepsilon\in\{0,1\}$$
such that for all $\zeta\in\hat{T},\varepsilon\in\{0,1\}$,
$$a_{\zeta,\varepsilon,0}\geq a_{\zeta,\varepsilon,1}\geq\dots\geq0,\quad\text{and}\quad \mu_\zeta\geq 0$$
and
$$ \sum_{{\underset{\varepsilon\in\{0,1\}}{\zeta\in\hat{T}}}}
\|a_{\zeta,\varepsilon}\|+\sum_{\zeta\in\hat{T}}\mu_\zeta\leq1$$
\noindent We define, for $\varepsilon\in\{0,1\}$, and $\zeta\in\hat{T}$ the characters $\chi_\varepsilon\colon \sigma\longmapsto  \text{sgn}(\sigma)^\varepsilon$ on $\mathfrak{S}_\infty$ and $\chi_\zeta\colon t\longmapsto \tr(\zeta(t))$ on $T$. For instance, in the trivial case $t=e_T$, we obtain $\chi_\zeta(t)=\tr(\id)=\dim\zeta$.

\begin{theorem}[Theorem 2, \cite{hirai}]\label{hirai+}
    $\mathbb{P}$ is in bijection with the set of all extremal characters of G. To a family of parameters $(a_{\zeta,\varepsilon})_{\zeta\in\hat{T},\varepsilon\in\{0,1\}},(\mu_\zeta)_{\zeta\in\hat{T}}$, we associate the character $f$ such that for all $g\in G$ with a standard decomposition such as in the Proposition \ref{standard}, we have :
    \begin{align*}
        f(g)=\prod_{1\leq k\leq r}&\left\{\sum_{\zeta\in\hat{T}}\left( \sum_{\varepsilon\in\{0,1\}}\sum_{i\in\mathbb{N}}\frac{a_{\zeta,\varepsilon,i}}{\dim\zeta}+\frac{\mu_\zeta}{\dim \zeta}
 \right)\chi_\zeta(t_{q_k})\right\}\\[3mm]
 \times\prod_{1\leq j\leq m}&\left\{\sum_{\zeta\in\hat{T}}\left( \sum_{\varepsilon\in\{0,1\}}\sum_{i\in\mathbb{N}}\left(\frac{a_{\zeta,\varepsilon,i}}{\dim\zeta}\right)^{\ell(\sigma_j)}\chi_\varepsilon(\sigma_j)
 \right)\chi_\zeta(P_{\sigma_j}(d_j))
 \right\}
    \end{align*}
\end{theorem}

This theorem can be seen as a generalization of the bijection described in the first section in the case of $\mathfrak{S}_\infty$. It is the main result of this paper and is proved in the next section.

\begin{theorem}\label{final}
    A family of parameters in $\mathbb{P}$ is associated to a Yang-Baxter extremal character by the bijection described in the Theorem \ref{hirai+} if and only if 
    \begin{itemize}
        \item Finitely many $a_{\zeta,\varepsilon,i} $ are non-zero.
        \item For all $\zeta\in\hat{T}$, $\mu_\zeta=0$
        \item $\sum_{\zeta,\varepsilon} \|a_{\zeta,\varepsilon}\|=1$
        \item For all $\zeta\in\hat{T}$, for all $\varepsilon\in\{0,1\}$, for all $i$, $a_{\zeta,\varepsilon,i} \in \mathbb{Q}_{\geq0}$
    \end{itemize}
\end{theorem}

\section{Proof of the Theorem \ref{final}}
\subsection{The necessary condition}

 The necessary condition is an application of the Theorem \ref{thoma}. Let $(\pi,R)$ be an extremal Yang-Baxter couple, and $\chi$ the extremal Yang-Baxter character associated to this couple. We write $\sigma$ instead of $(1,\sigma)\in G$. We recall that $c_n=\sigma_{1}\cdots\sigma_{n-1}$. By construction, $\chi_{|\mathfrak{S}_\infty}$ is an (extremal) Yang-Baxter character. However,
    $$\chi(c_n)=\sum_{\zeta\in\hat{T}} \sum_{\varepsilon\in\{0,1\}}\sum_{i\in\mathbb{N}}\left(\frac{a_{\zeta,\varepsilon,i}}{\dim\zeta}\right)^{n}\chi_\varepsilon(c_n)(\dim\zeta)$$
    \noindent We can write it as :
    $$\chi(c_n)=\sum_{\zeta\in\hat{T}} \sum_{i\in\mathbb{N}}\left(\frac{a_{\zeta,0,i}}{\dim\zeta}\right)^{n}(\dim\zeta) +(-1)^{n-1}\left(\frac{a_{\zeta,1,i}}{\dim\zeta}\right)^{n}(\dim\zeta)$$
    \noindent And the Thoma parameters associated to $\chi_{|\mathfrak{S}_\infty}$ are $\alpha,\beta$ with
    $$\alpha = \bigcup_{\underset{1\leq\ell\leq\dim\zeta}{\zeta\in\hat{T}}}\frac{a_{\zeta,0}^{(\ell)}}{\dim\zeta},\quad\text{et} \quad \beta=\bigcup_{\underset{1\leq\ell\leq\dim\zeta}{\zeta\in\hat{T}}}\frac{a_{\zeta,1}^{(\ell)}}{\dim\zeta}$$
    \noindent with $a_{\zeta,\varepsilon}^{(\ell)}$ being a copy of $a_{\zeta,\varepsilon}$. According to the theorem \ref{thoma}, we obtain that finitely many $a_{\zeta,\varepsilon}$ are non-zero, and there exists $d\in\mathbb{N}$ such that for all $\zeta\in\widehat{T}$, $\varepsilon\in\{0,1\}$, for all $i$, 
    $$d\frac{a_{\zeta,\varepsilon,i}}{\dim\zeta}\in\mathbb{N}$$
    Hence the first and the fourth conditions. We also obtain the following equality $$\sum_{\underset{1\leq\ell\leq\dim\zeta}{\zeta\in\hat{T}}} \left\|\frac{a_{\zeta,0}^{(\ell)}}{\dim\zeta}\right\| +\left\|\frac{a_{\zeta,1}^{(\ell)}}{\dim\zeta}\right\|=1$$
    \noindent And by definition of $a_{\zeta,\varepsilon}^{(\ell)}$, 
    $$\sum_{\zeta\in\hat{T}} \left\|\frac{a_{\zeta,0}}{\dim\zeta}\right\|(\dim\zeta) +\left\|\frac{a_{\zeta,1}}{\dim\zeta}\right\|(\dim\zeta)=\sum_{\zeta\in\hat{T}} \left\|a_{\zeta,0}\right\| +\left\|a_{\zeta,1}\right\|=1$$
    Nevertheless, by the definition of the set of parameters,
   
       $$ 1\geq \sum_{{\underset{\varepsilon\in\{0,1\}}{\zeta\in\hat{T}}}}
\|a_{\zeta,\varepsilon}\|+\sum_{\zeta\in\hat{T}}\mu_\zeta\geq\sum_{{\underset{\varepsilon\in\{0,1\}}{\zeta\in\hat{T}}}}
\|a_{\zeta,\varepsilon}\|=1.$$
    Thus all inequalities are equalities, and for all $\zeta\in\hat{T}$, $\mu_\zeta=0$. 
    
    \noindent Hence the necessary conditions.

\subsection{The sufficient condition}

To prove the sufficient condition, we need to construct an explicit Yang-Baxter couple corresponding to a given family of parameters. Take $W=\mathbb{C}$. Let $(a_{\zeta,\varepsilon})_{\zeta,\varepsilon},\mu$ be parameters and $d\in\mathbb{N}$ verifying the conditions described in Theorem \ref{final}. Let $V_{\zeta,\varepsilon,i}$ and $W_{\zeta,\varepsilon,i}$ be two (complex) Hilbert spaces of dimension $\dim\zeta$ and $d_{\zeta,\varepsilon,i}:=d\frac{a_{\zeta,\varepsilon,i}}{\dim\zeta}$ for all $\zeta\in\hat{T}$, $\varepsilon\in\{0,1\}$, $i\in\mathbb{N}$ such that
$$V=\bigoplus_{\zeta,\varepsilon,i} V_{\zeta,\varepsilon,i}\otimes W_{\zeta,\varepsilon,i}$$
\noindent For all $\zeta\in\hat{T},\varepsilon\in\{0,1\},i\in\mathbb{N}$, let $R_{\zeta,\varepsilon,i}\in End\left((V_{\zeta,\varepsilon,i}\otimes W_{\zeta,\varepsilon,i})^{\otimes2}\right)$ verifying for all $v,v'\in V_{\zeta,\varepsilon,i}$ and $w,w'\in W_{\zeta,\varepsilon,i}$ :
$$R_{\zeta,\varepsilon,i} (v\otimes w\otimes v'\otimes w')= (-1)^\varepsilon v'\otimes w\otimes v\otimes w'$$
In other words, $R_{\zeta,\varepsilon,i}$ flips the $V_{\zeta,\varepsilon,i}$ components and is $\pm$id on the $W_{\zeta,\varepsilon,i}$ components.

\begin{lemma}
$R_{\zeta,\varepsilon,i}$ is an $R$-matrix with Thoma parameters $\alpha_1=\alpha_2=\cdots=\alpha_{\dim\zeta}=\frac{1}{\dim\zeta}$, $\beta=0$ if $\varepsilon=0$ and $\alpha=0$, $\beta_1=\beta_2=\cdots=\beta_{\dim\zeta}=\frac{1}{\dim\zeta}$ if $\varepsilon=1$.
\end{lemma}

\begin{proof}
We assume $\varepsilon=0$ to avoid heavy notations but the other case is analogous. Let $v_1,v_2,v_3\in V_{\zeta,\varepsilon,i} $ and $w_1,w_2,w_3\in W_{\zeta,\varepsilon,i}$.
\begin{align*}
(R\otimes 1)(1\otimes R)(R&\otimes 1)(v_1\otimes w_1)\otimes(v_2\otimes w_2)\otimes(v_3\otimes w_3)\\[2mm]
&=(R\otimes 1)(1\otimes R)(v_2\otimes w_1)\otimes(v_1\otimes w_2)\otimes(v_3\otimes w_3)\\[4mm]
&=(R\otimes 1)(v_2\otimes w_1)\otimes(v_3\otimes w_2)\otimes(v_1\otimes w_3)\\[4mm]
&=(v_3\otimes w_1)\otimes(v_2\otimes w_2)\otimes(v_1\otimes w_3)
\end{align*}
\noindent and 
\begin{align*}
(1\otimes R)(R\otimes 1)(1&\otimes R)(v_1\otimes w_1)\otimes(v_2\otimes w_2)\otimes(v_3\otimes w_3)\\[2mm]
&=(1\otimes R)(R\otimes 1)(v_1\otimes w_1)\otimes(v_3\otimes w_2)\otimes(v_2\otimes w_3)\\[4mm]
&=(1\otimes R)(v_3\otimes w_1)\otimes(v_1\otimes w_2)\otimes(v_2\otimes w_3)\\[4mm]
&=(v_3\otimes w_1)\otimes(v_2\otimes w_2)\otimes(v_1\otimes w_3)
\end{align*}

\noindent Let $n\geq1$. If $A:= V_{\zeta,\varepsilon,i}\otimes  W_{\zeta,\varepsilon,i}$, and $v_1,\dots,v_n\in V_{\zeta,\varepsilon,i}$, and $w_1,\dots w_n\in W_{\zeta,\varepsilon,i}$. We assume that $v_1,\dots v_n$ are in an orthonormal set.

\begin{align*}
(R\otimes1_A^{\otimes n-1})&(1_A\otimes R\otimes1_A^{\otimes n-2})\dots (1_A^{\otimes n-1}\otimes R)(v_1\otimes w_1)\otimes\dots\otimes(v_n\otimes w_n)\\[2mm]
&\not\perp (v_1\otimes w_1)\otimes\dots\otimes(v_n\otimes w_n)\quad \Longleftrightarrow \forall i,j \quad v_i\not\perp v_j\Longleftrightarrow\forall i,j \quad v_i= v_j
\end{align*}
We recall that $\dim V_{\zeta,\varepsilon,i}=\dim\zeta $ and $\dim W_{\zeta,\varepsilon,i}=d_{\zeta,\varepsilon,i}$. This implies 
$$\tr((R\otimes1_A^{\otimes n-1})(1_A\otimes R\otimes1_A^{\otimes n-2})\dots (1_A^{\otimes n-1}\otimes R))=\sum_{k=1}^{\dim\zeta}(\dim W_{\zeta,\varepsilon,i})^n$$
\noindent and  
$$\tau(\yb(\sigma_1\dots\sigma_{n-1}))=\frac{1}{(\dim\zeta)^n{d_{\zeta,\varepsilon,i}}^n}(\dim\zeta) {d_{\zeta,\varepsilon,i}}^n=\frac{1}{(\dim\zeta)^{n-1}}$$
Thus, the Thoma parameters are 
$$\alpha_1=\alpha_2=\cdots=\alpha_{\dim\zeta}=\frac{1}{\dim\zeta}$$
\end{proof}

We define the following normal form :
$$R:=\underset{\zeta,\varepsilon,i}{\boxplus}R_{\zeta,\varepsilon,i}$$
\begin{remark}
Let $V=\oplus_i V_i$, and $R_i\in End(V_i\otimes V_i)$ for all $i$ an $R$-matrix with Thoma parameters $(\{\alpha_j^{(i)}\}_j,\{\beta_j^{(i)}\}_j)$.
Let $d_i=\dim V_i$ and $d=\dim V$. Then, the Thoma parameters of the $R$-matrix $R=\boxplus_i R_i$ are $$\left(\left\{\frac{d_i}{d}\alpha_j^{(i)}\right\}_{i,j}~,~\left\{\frac{d_i}{d}\beta_j^{(i)}\right\}_{i,j}\right)$$
\end{remark}

Thus, in our case, $R_{\zeta,\varepsilon,i}\in End(E_{\zeta,\varepsilon,i}\otimes E_{\zeta,\varepsilon,i})$ with $E_{\zeta,\varepsilon,i}:=V_{\zeta,\varepsilon,i}\otimes W_{\zeta,\varepsilon,i}$ and $\dim E_{\zeta,\varepsilon,i}=da_{\zeta,\varepsilon,i}$. So the Thoma parameters of the normal form $R$-matrix $R$ are 
$$\alpha=\bigcup_{\zeta,i,\ell} \left\{  \frac{a_{\zeta,0,i,\ell}}{\dim\zeta} \right\}\qquad\qquad\beta=\bigcup_{\zeta,i,\ell} \left\{  \frac{a_{\zeta,1,i,\ell}}{\dim\zeta} \right\}$$
\noindent with $a_{\zeta,\varepsilon,i,\ell}$ a copy of $a_{\zeta,\varepsilon,i}$ for $1\leq\ell\leq\dim\zeta$.

Let $\pi$ be the representation of $T$ such that for all $\zeta,\varepsilon,i$, for all $t\in T$
$$\pi(t)_{|E_{\zeta,\varepsilon,i}}=\zeta(t)\otimes \id_{W_{{\zeta,\varepsilon,i}}}$$

\paragraph{Extremal Yang-Baxter couple}Let us show that $$R(\pi\otimes1) R=1\otimes\pi$$
\noindent This will prove that $(\pi,R)$ is an extremal Yang-Baxter couple. Let $t\in T$, let $x\otimes y\in E_{\zeta,\varepsilon,i}\otimes E_{\zeta',\varepsilon',j}$ with $(\zeta,\varepsilon,i)\neq (\zeta',\varepsilon',j)$.
$$R(\pi(t)\otimes1) R(x\otimes y)=R(\pi(t)\otimes1) (y\otimes x)=R\left[\left(\zeta(t)\otimes\id~ y\right)\otimes x\right]=(1\otimes\pi(t))(x\otimes y)$$
\noindent Let $(v_1\otimes w_1)\otimes (v_2\otimes w_2)\in E_{\zeta,\varepsilon,i}\otimes E_{\zeta,\varepsilon,i}$.
\begin{align*}
R(\pi(t)\otimes1) R(v_1\otimes w_1\otimes v_2\otimes w_2)&=R_{\zeta,\varepsilon,i}(\pi(t)\otimes1) R_{\zeta,\varepsilon,i}(v_1\otimes w_1\otimes v_2\otimes w_2)\\[3mm]
&=R_{\zeta,\varepsilon,i}(\pi(t)\otimes1)(v_2\otimes w_1\otimes v_1\otimes w_2)\\[3mm]
&=R_{\zeta,\varepsilon,i} (\zeta(t)v_2\otimes w_1\otimes v_1\otimes w_2)\\[3mm]
&=v_1\otimes w_1\otimes \zeta(t)v_2\otimes w_2\\[2mm]
&=(1\otimes\pi(t))(v_1\otimes w_1\otimes v_2\otimes w_2)
\end{align*}
Hence the extremality of the Yang-Baxter couple $(\pi,R)$.

\paragraph{Parameters} Let $\rho$ be the Yang-Baxter representation of $G$ associated to $(\pi,R)$, and $\chi$ its character. We already know that for $n\in\mathbb{N}$,
$$\chi(\sigma_1\sigma_2\cdots\sigma_{n-1})=\sum_{\zeta\in\hat{T}} \sum_{i\in\mathbb{N}}\left(\frac{a_{\zeta,0,i}}{\dim\zeta}\right)^{n}(\dim\zeta) +(-1)^{n-1}\left(\frac{a_{\zeta,0,i}}{\dim\zeta}\right)^{n}(\dim\zeta)$$
We saw that every cyclic or elementary element is conjugate in $G$ to an element of the form $g=(d,\sigma)\in G$ with $\sigma=(1\cdots n)$ a cyclic permutation and $d=(t,1,1,\cdots)$. Let $g$ be this element. To simplify the notations, we will write $\pi(t)$ instead of $(\pi(t)\otimes1^{\otimes n-1})$.
$$\left( \pi(t)R_1\cdots R_{n-1} \right)E_{\zeta_1,\varepsilon_1,i_1}\otimes\cdots\otimes E_{\zeta_n,\varepsilon_n,i_n}\not\perp E_{\zeta_1,\varepsilon_1,i_1}\otimes\cdots\otimes E_{\zeta_n,\varepsilon_n,i_n}$$
$$\Longleftrightarrow (\zeta_1,\varepsilon_1,i_1)=(\zeta_2,\varepsilon_2,i_2)=\cdots=(\zeta_n,\varepsilon_n,i_n)$$
\noindent and for $v_1\otimes w_1\otimes\cdots\otimes v_n\otimes w_n\in E_{\zeta,\varepsilon,i}^{\otimes n} $, if $\varepsilon=0$ then
$$\left( \pi(t)R_1\cdots R_{n-1} \right)v_1\otimes w_1\otimes\cdots\otimes v_n\otimes w_n=\zeta(t)v_n\otimes w_1\otimes v_1\otimes w_2 \otimes\cdots\otimes v_{n-1}\otimes w_n$$
\noindent and if $\varepsilon=1$ then 
$$\left( \pi(t)R_1\cdots R_{n-1} \right)v_1\otimes w_1\otimes\cdots\otimes v_n\otimes w_n=(-1)^{n-1}\zeta(t)v_n\otimes w_1\otimes\cdots\otimes v_{n-1}\otimes w_n$$
\noindent Thus,
$$\tr\left( \pi(t)R_1\cdots R_{n-1} \right)=\sum_{\zeta,\varepsilon,i}{d_{\zeta,\varepsilon,i}}^n\chi_\varepsilon(\sigma)\tr(\zeta(t))$$
\noindent and 
\begin{align*}
\chi(g)&=\frac{1}{d^n}\tr\left( \pi(t)R_1\cdots R_{n-1} \right)\\[3mm]
&=\frac{1}{d^n}\sum_{\zeta,\varepsilon,i}\frac{d^n{a_{\zeta,\varepsilon,i}}^n}{\dim_\zeta^n}\chi_\varepsilon(\sigma)\tr(\zeta(t))\\[3mm]
&=\sum_{\zeta\in\hat{T}}\left\{\sum_{\varepsilon\in\{0,1\}}\sum_{i\in\mathbb{N}}\left(\frac{{a_{\zeta,\varepsilon,i}}}{\dim_\zeta}\right)^n\chi_\varepsilon(\sigma)\right\}\chi_\zeta(t)
\end{align*}
Finally, for $g\in G$ with the standard decomposition such as in the proposition \ref{standard}, we get by extremality
\begin{align*}
        \chi(g)=\prod_{1\leq k\leq r}&\left\{\sum_{\zeta\in\hat{T}}\left( \sum_{\varepsilon\in\{0,1\}}\sum_{i\in\mathbb{N}}\frac{a_{\zeta,\varepsilon,i}}{\dim\zeta}
 \right)\chi_\zeta(t_{q_k})\right\}\\[3mm]
 \times\prod_{1\leq j\leq m}&\left\{\sum_{\zeta\in\hat{T}}\left( \sum_{\varepsilon\in\{0,1\}}\sum_{i\in\mathbb{N}}\left(\frac{a_{\zeta,\varepsilon,i}}{\dim\zeta}\right)^{\ell(\sigma_j)}\chi_\varepsilon(\sigma_j)
 \right)\chi_\zeta(P_{\sigma_j}(d_j))
 \right\}
    \end{align*}

This proves the Theorem \ref{final}.

\vskip 3cm
\subsubsection*{Acknowledgements}
I would like to thank Kenny de Commer for his help, advice and numerous reviews, and Vrije Universiteit Brussels for its hospitality.

\newpage

\bibliographystyle{alpha}
\bibliography{bibiblio}

@article{lechner2019yang,
  title={Yang-Baxter representations of the infinite symmetric group},
  author={Lechner, Gandalf and Pennig, Ulrich and Wood, Simon},
  journal={Advances in Mathematics},
  volume={355},
  pages={106769},
  year={2019},
  publisher={Elsevier}
}

@article{hirai,
  title={Characters of wreath products of finite groups with the infinite symmetric group},
  author={Hirai, Takeshi and Hirai, Etsuko},
  journal={Journal of Mathematics of Kyoto University},
  volume={45},
  number={3},
  pages={547--597},
  year={2005},
  publisher={Duke University Press}
}

@article{thoma1964unzerlegbaren,
author = {Thoma, Elmar},
journal = {Mathematische Zeitschrift},
pages = {40-61},
title = {Die unzerlegbaren, positiv-definiten Klassenfunktionen der abzählbar unendlichen, symmetrischen Gruppe.},
url = {http://eudml.org/doc/170284},
volume = {85},
year = {1964},
}

@article{Schwiebert94,
   title={Extended reflection equation algebras, the braid group on a handlebody, and associated link polynomials},
   volume={35},
   ISSN={1089-7658},
   url={http://dx.doi.org/10.1063/1.530751},
   DOI={10.1063/1.530751},
   number={10},
   journal={Journal of Mathematical Physics},
   publisher={AIP Publishing},
   author={Schwiebert, Christian},
   year={1994},
   month=oct, pages={5288–5305} }

@article{Lechner_2007,
   title={Construction of Quantum Field Theories with Factorizing S-Matrices},
   volume={277},
   ISSN={1432-0916},
   url={http://dx.doi.org/10.1007/s00220-007-0381-5},
   DOI={10.1007/s00220-007-0381-5},
   number={3},
   journal={Communications in Mathematical Physics},
   publisher={Springer Science and Business Media LLC},
   author={Lechner, Gandalf},
   year={2007},
   month=nov, pages={821–860} }

@misc{faddeev,
      title={How Algebraic Bethe Ansatz works for integrable model}, 
      author={L. D. Faddeev},
      year={1996},
      eprint={hep-th/9605187},
      archivePrefix={arXiv},
      primaryClass={hep-th},
      url={https://arxiv.org/abs/hep-th/9605187}, 
}

@article{cherednik,
  title={Factorizing particles on a half-line and root systems},
  author={Cherednik, Ivan Vladimirovich},
  journal={Teoreticheskaya i Matematicheskaya Fizika},
  volume={61},
  number={1},
  pages={35--44},
  year={1984},
  publisher={Russian Academy of Sciences, Steklov Mathematical Institute of Russian~…}
}

@book{borodin,
  title={Representations of the infinite symmetric group},
  author={Borodin, Alexei and Olshanski, Grigori},
  volume={160},
  year={2017},
  publisher={Cambridge University Press}
}

\end{document}